\documentclass[a4paper]{article}
\usepackage{hyperref}
\usepackage{amsmath}
\usepackage[dvips]{graphicx}

\usepackage{amsthm}
\usepackage{epsfig}
\usepackage{amssymb}
\usepackage{dsfont}
\usepackage{float}
\usepackage{color}

\newtheorem{theorem}{Theorem}[section]
\newtheorem{lemma}[theorem]{Lemma}

\newtheorem{rem}{Remark}[section]
\newtheorem{de}{Definition}[section]

\begin{document}

\begin{center}
{\Large  Stability of periodic peakons for the Novikov equation}
\end{center}
\vskip 5mm

\begin{center}
{\sc Yun Wang and Lixin Tian} \\
\vskip2mm
Institute of Mathematics,
School of Mathematical Sciences,\\
Nanjing Normal University,
Nanjing, 210023, China
\vskip 5mm
\end{center}

\vskip 5mm {\leftskip5mm\rightskip5mm \normalsize
\noindent{\bf{Abstract}} 
The Novikov equation is an integrable Camassa-Holm type equation with cubic nonlinearity and admits 
the periodic peakons. In this paper, it is shown that the periodic peakons are the global periodic weak 
solutions to the Novikov equation and we also prove the orbital stability of the periodic peakons for
Novikov equation. By using the invariants of the equation and controlling the extrema of the solution,
it is demonstrated that the shapes of these periodic peakons are stable under small perturbations in the 
energy space. 

}

\newtheorem{proposition}[theorem]{Proposition}

\renewcommand{\theequation}{\thesection.\arabic{equation}}
\catcode`@=11
\@addtoreset{equation}{section}
\catcode`@=12

\section{Introduction}

We consider here the orbital stability of the periodic peakons for the
Novikov equation in the form
\begin{equation}\label{NV}
   u_{t}-u_{xxt}+4u^{2}u_{x}-3uu_{x}u_{xx}-u^{2}u_{xxx}=0, \quad x\in \mathbb{R},\quad t>0,
\end{equation}
for the function $u(x,t)$ of a single spatial variable $x$ and time $t$. Equation (\ref{NV})
was derived by Novikov $\cite{GNV}$ in a symmetry classification of nonlocal partial
differential equations with cubic nonlinearity. By defining a new dependent variable $y$,
equation (\ref{NV}) can be written as a compact form
\begin{equation}\label{CNV}
    y_{t}+u^2y_{x}+\frac{3}{2}(u^2)_{x}y=0, \quad y=u-u_{xx},
\end{equation}
which can be regarded as a generalization for the celebrated Camassa-Holm (CH) equation
$\cite{SWP}$ or the Degasperis-Procesi (DP) equation $\cite{DP}$.

The celebrated Camassa-Holm (CH) equation
\begin{equation}\label{CH}
    y_{t}+uy_{x}+2u_{x}y=0, \quad y=u-u_{xx},
\end{equation}
was proposed as a model for the unidirectional propagation of the shallow water waves
over a flat bottom\cite{AD, SWP, CHR}, with $u(x,t)$ representing the water's
free surface in nondimentional variables \cite{SWP}. It can be found using the method of recursion
operators as an example of bi-Hamiltonian equation with an infinite number of conserved
functionals by Fokas and Fuchssteiner \cite{BF}. The CH equation has attracted much attention
in the last two decades because of its interesting properties: complete integrability
\cite{BF, SWP}, existence of peaked solitons and multi-peakons \cite{SWP}, geometric
formulations and the presence of breaking waves(i.e. a wave profile remains bounded while its
slope becomes unbounded in finite time). Among these properties, a remarkable one is that it
admits the single peakons and periodic peakons in the following forms
\begin{equation}\label{SP}
    \varphi_{c}(x,t)=ce^{-|x-ct|}, \quad c\in \mathbb{ R},
\end{equation}
and
\begin{equation}\label{PP}
    u_{c}(x,t)=\frac{c}{\operatorname{sh}(\frac{1}{2})}\operatorname{ch}(\frac{1}{2}-(x-ct)+[x-ct]), 
    \quad c\in \mathbb{ R},
\end{equation}
here the notation $[x]$ denotes the largest integer part of the real number $x\in\mathbb{R}$,
and the multi-peakon solutions
\begin{equation}\label{MP}
    u(x,t)=\sum^{N}_{i=1}p_{i}(t)e^{-|x-q_{i}(t)|},
\end{equation}
where $p_{i}(t)$ and $q_{i}(t)$ satisfy the Hamiltonian system
$$
\left \{
   \begin{array}{l}
      \dot{q_{i}}=\displaystyle\sum\limits_{j}p_{j}e^{-|q_{i}-q_{j}|}=\frac{\partial H}{\partial p_{i}},\\
      \dot{p_{i}}=\displaystyle\sum\limits_{j\neq i}p_{i}p_{j}sign(q_{i}-q_{j})e^{-|q_{i}-q_{j}|}=
      -\frac{\partial H}{\partial q_{i}},
   \end{array}
   \right.
$$
with the Hamiltonian
$$
H=\frac{1}{2}\sum^{N}_{i,j=1,2}p_{i}p_{j}e^{-|q_{i}-q_{j}|}.
$$
The solutions consisting of a train of infinite many peaked solitary waves were
established in $\cite{DDCS}$. Note that the peakons are smooth solutions of (\ref{CH})
except at the peak points $x=q_{i}(t)$, where the derivative of $u$ is discontinuous.
The interest in peakons is great because they are relatively new forms of solitary
waves (for most models the solitary waves are quite smooth). More importantly, in the
theory of water waves many papers have investigated the Stokes waves of greatest height,
traveling waves which are smooth everywhere except at the crest where the lateral
tangents differ. There is no closed form available for these waves and the peakons
capture the essential features of the extreme waves (shown in\cite{A}). To validate
the peakons as physically relevant solutions of the CH or DP models, it does not
suffice to show their existence (readily obtained in explicit form).
It is necessary to show that these wave patterns can
be detected, namely, their shape is stable under perturbations (orbital stability).
In an intriguing paper by Constantin and Strauss\cite{AL1}, it was shown that the
single peakons for the CH equation are orbitally stable. The approach of proof is to
use the energy as a Lyapunov functional and derive stability in view of the way that
certain conservation laws interact. A variational approach for proving the orbital
stability of the peaked solitons was introduced by Constantin and Molient\cite{AL}.
Orbital stability of multi-peakon solutions was proved by Dika and Molient\cite{DL}
using the approaches in \cite{AL1}. Stability of the periodic peaked solitons for the
CH equation was proved by Lenells\cite{LJ}. The approach in\cite{LJ}was extended
in \cite{JR} to prove the orbital stability of the periodic peaked solitons $u_{c}(x,t)=c
\varphi(x-ct)$ with $\varphi(x)=\displaystyle\frac{1}{26}(12x^2+23)$, \ $\displaystyle 
x\in[-1/2,1/2]$ for the $\mu$-CH equation
$$
y_{t}+uy_{x}+2u_{x}y=0, \quad y=\mu(u)-u_{xx},
$$
where $\mu(u)=\displaystyle \int_{S^1}u(x,t)dx$, with $S^1=\mathbb{R}/\mathbb{Z}$
and $\varphi$ is extended periodically to the real line.

The DP equation
$$
y_{t}+uy_{x}+3yu_{x}=0,\quad y=u-u_{xx},
$$
was obtained in the study of asymptotical integrability for a family of three-order
nonlinear dispersive evolution equations\cite{DP}. It can also be derived as a member
of a one-parameter family of asymptotic shallow water wave approximations to the Euler
equations with the same asymptotic accuracy as that of the CH equation\cite{AD}.
Similarly to the CH equation, the DP equation also has the peaked solitons\cite{ADD},
and is integrable with the Lax pair and bi-Hamiltonian structure\cite{DP}.
It is noticed that the Lax pair for the DP equation and the CH equation are the
$3\times3$ and $2\times2$ spectral problems, respectively. The orbital stability of the
single peakons for the DP equation was proved by Lin and Liu\cite{ZWL}. They developed
the approach due to Constantin and Strauss\cite{AL1} in a delicate way.

It is observed that all nonlinear terms in the CH and DP equations are quadratic. In
contract to the integrable modified KdV equation with a cubic nonlinearity, it is of great
interest to find integrable CH-type equations with cubic or higher-order nonlinearity
admitting peaked solitons. Up to now, to the best of our knowledge, two integrable CH-type
equations with cubic nonlinearities have been discovered. One was introduced by Olver
and Rosenau\cite{PJP} by using the tri-Hamiltonian duality approach, which is the so-called
modified CH equation in the form
$$
y_{t}+[(u^2-u^2_{x})y]_{x}=0, \quad y=u-u_{xx},
$$
admits the single peakon \cite{GGY}
$$
\varphi_{c}(x,t)=\sqrt{\frac{3c}{2}}e^{-|x-ct|},\quad c>0
$$
and the periodic peakon
$$
u_{c}(x,t)=\sqrt{\frac{3c}{1+2\operatorname{ch}^2(\frac{1}{2})}}\operatorname{ch}(\frac{1}{2}
-(x-ct)+[x-ct]),\quad c>0.
$$
The issue of the stability of peakons for the modified CH equation is under investigation\cite{Qu}.
The other one is the Novikov equation (\ref{NV}). Both equations were shown to have some new 
properties which are different from that for the CH and DP equations\cite{FT, GGY}. 
This paper will only focus on the Novikov equation.

It is known that the Novikov equation (\ref{NV}) is integrable with the Lax pair\cite{GNV}
$$
\begin{array}{ll}
\displaystyle\psi_{xxx}=\psi_{x}+\lambda y^2\psi+2\frac{y_{x}}{y}\psi_{xx}+\frac{yy_{xx}
-2y^2_{x}}{y^2}\psi_{x},\\[3mm]
\displaystyle\psi_{t}=\frac{u}{\lambda y}\psi_{xx}-\frac{yu_{x}+uy_{x}}{y^2}\psi_{x}-u^2\psi_{x}.
\end{array}
$$
A matrix Lax pair representation to the Novikov equation was recently provided by Hone and Wang\cite{WAN}.
With that representation it can be shown that the Novikov equation is related to a negative flow in the
Sawada-Kotera hierarchy. It is also noticed that the Novikov equation admits a bi-Hamiltonian structure
\cite{WAN} and it can be written as
$$
y_{t}=\mathcal{B}_{1}\frac{\delta H_{1}}{\delta y}=\mathcal{B}_{2}\frac{\delta H_{2}}{\delta y}
$$
with the Hamiltonian operators
$$
\begin{array}{ll}
\mathcal{B}_{1}=-2(3y\partial_{x}+2y_{x})(4\partial_{x}-\partial_{x}^3)^{-1}(3y\partial_{x}+y_{x}),\\[3mm]
\mathcal{B}_{2}=(1-\partial_{x}^2)\frac{1}{y}\partial_{x}\frac{1}{y}(1-\partial_{x}^2),
\end{array}
$$
and the corresponding Hamiltonians
$$
\begin{array}{ll}
H_{1}=\displaystyle \int(u^2+u^2_{x})dx,\\[3mm]
H_{2}=\displaystyle\frac{1}{6} \int uy\partial^{-1}y(\partial^{2}-1)^{-1}(u^2y_{x}+3uu_{x}y)dx.
\end{array}
$$

An interesting property of the Novikov equation (\ref{NV}), which is common with the CH and DP
equations, is the existence of the peaked solitons, which is given explicitly in the form
\begin{equation}\label{ss}
u(x,t)=\varphi_{c}(x-ct)=\sqrt{c}\varphi(x-ct),
\end{equation}
traveling at constant speed $c>0$, where $\varphi(x)=e^{-|x|}$. Note that $e^{-|x|}/2$ is the convolution
kernel for the operator $(1-\partial^2_{x})^{-1}$. Hone, Lundmark and Szmigielski \cite{AHJ} obtained
multi-peakons of the Novikov equation explicitly by using the inverse scattering approach.
The peaked solitons (\ref{ss}) are not classical solutions but weak solutions of (\ref{NV}) 
in the conservation law form
\begin{equation}\label{1.8}
u_{t}+u^2u_{x}+p*(3uu_{x}u_{xx}+2u^3_{x}+3u^2u_{x})=0, \quad x\in \mathbb{R},\quad t>0,
\end{equation}
where $*$ stands for convolution with respect to the spatial variable $x\in \mathbb{R}$ and $p=\frac{1}{2}e^{-|x|}$.
The orbital stability of peakons on the line was shown in\cite{XLIU}.
In this paper, we only consider the stability of the following periodic peakons
\begin{equation}\label{1.9}
u(x,t)= \varphi_{c}(x,t)=\sqrt{c}\varphi(x-ct),
\end{equation}
where $\varphi(x)$ is given by
$$
\varphi(x)=\frac{\operatorname{ch}(\frac{1}{2}-x+[x])}{\operatorname{ch}(\frac{1}{2})},
$$
with  $\operatorname{ch}(x)=\displaystyle\frac{e^x+e^{-x}}{2}$. The main result is as follows.
\begin{theorem}
The periodic peakons (\ref{1.9}) of the equation (\ref{NV}) are orbitally stable in the energy space.
\end{theorem}
The following are three useful conservation laws of the Novikov equation\cite{WAN}
$$
\begin{array}{ll}
E_{1}(u)=\displaystyle \int(u-u_{xx})^\frac{2}{3}dx, \ \quad E_{2}(u)=\displaystyle \int(u^2+u_{x}^2)dx,\\[3mm]
E_{3}(u)=\displaystyle \int(u^4+2u^2u^2_{x}-\frac{1}{3}u_{x}^4)dx,
\end{array}
$$
which will play a key role in proving the orbital stability of the periodic peaked solutions, while the corresponding
three conservation laws of the CH equation are in the following:
$$
F_{1}(u)=\displaystyle \int(u-u_{xx})dx,\quad F_{2}(u)=E_{2}(u), \quad F_{3}(u)=\displaystyle 
\int(u^3+uu_{x}^2)dx.
$$
Due to the conservation law $E_{2}$, we can expect the orbital stability of periodic peakons for the Novikov 
equation.  It is found that the conservation law $E_{3}$ of the Novikov equation is much more complicated than 
$F_{3}$ of the CH equation. Therefore, the stability discussion of the periodic peakons for the Novikov equation 
is more difficult.  Our approach is inspired by \cite{Qu} where the stability of peakons(on the line or periodic) for 
the $m$-CH equation was studied and by \cite{LJ} where the stability of periodic peakons for the CH equation was 
also studied. The proof of stability is expored by finding appropriate inequalities related to the two conservation laws 
$E_{2}$ and $E_{3}$ with global maximum and minimum of the solution.  To establish these inequalities in the present 
case one requires to make use of two special ingredients: i) the conserved equalities $E_{2}$ and $E_{3}$, ii) two 
functionals with global maximum and minimum of the solution, which are connected to the two conservation laws 
$E_{2}$ and $E_{3}$. On the other hand, the two functionals are required to vanish at the  periodic peakons $\varphi_{c}$. 
The observation is crucial since the periodic peakons, if stable, must be critical points of the energy functional with 
momentum constraint, and should satisfy the corresponding Euler-Lagrangian equations.

$Notation.$  Throughout the paper, the norm of a Banach space $Z$ is denoted by $\|\cdot\|_{Z}$. 
In the periodic case, we denote $\mathbb{S}=\mathbb{R}/\mathbb{Z}$ as the unit circle and 
regard the function on $\mathbb{S}$ as a periodic function on the entire line with period one. Given $T>0$, let 
$C^{\infty}_{c}(\mathbb{S}\times[0,T))$ denote the space of all smooth functions with compact support on $\mathbb{S}\times[0,T)$, 
which can  be obtained as the restriction to $\mathbb{S}\times[0,T)$ of smooth functions on $\mathbb{S}\times\mathbb{R}$ 
with compact support contained in $\mathbb{R}\times[0,T)$. 
Let $*$ denote convolution with respect to the spatial variable $x\in\mathbb{S}$. 
For any Banach space $Z$ and real number $T>0$, $C([0,T),Z)$ 
is the class  of continuous functions from $[0,T)$ to $Z$.
For an integer $n\geq1$, we let $H^n(\mathbb{S})$ denote the Sobolev space of all square integrable functions $f\in L^2(\mathbb{S})$ 
with distributional derivatives $\partial^i_{x}f\in L^2(\mathbb{S})$  for $i=1,2, ..., n. $ The norm on $H^n(\mathbb{S})$ is given by 
 $$
 \|f\|^2_{H^n(\mathbb{S})}=\sum^n_{i=0}\int_{\mathbb{S}}(\partial^i_{x}f)^2(x)dx.
 $$

\section{Preliminaries}

In this section,  we consider the Cauchy problem for the Novikov equation on the unit circle:
\begin{equation}\label{2.1}
\left \{
   \begin{array}{l}
   y_{t}+u^2y_{x} +3uu_{x}y=0, \quad y=u-u_{xx},\quad t>0, \quad x\in \mathbb{S},\\
   u(0,x)=u_{0}(x),\quad x\in \mathbb{S}.
   \end{array}
   \right.
\end{equation}
Since all space of functions are over $\mathbb{S}$, for simplicity, we drop $\mathbb{S}$ in our notations 
of function spaces if there is no ambiguity. First, we give the notion of strong solutions as follows.
\begin{de}\label{de2.1}
If $u\in C([0, T); H^s(\mathbb{S}))\cap C^1([0, T); H^{s-1}(\mathbb{S}))$ with $s>\frac{3}{2}$ and some $T>0$ 
satisfies (\ref{2.1}),  then $u$ is  called a strong solution on $[0,T)$.  If $u$ is a strong solution on $[0,T)$ for 
every $T>0$, then it is called a global strong solution.
\end{de}
The following local well-posedness result of the Cauchy problem(\ref{2.1}) for strong solutions
on the unit circle  with initial data $u_{0}\in H^s$, $s>3/2$, can be obtained by applying a 
Galerkin-type approximation method which is established by Himonas and Holliman in \cite{AC}.
\begin{lemma}(Local well-posedness)
If $s>3/2$ and $u_{0}\in H^s(\mathbb{S})$, then there exists $T>0$ such that the Cauchy problem(\ref{2.1}) has 
a unique strong solution  $u\in C([0, T); H^s(\mathbb{S}))\cap C^1([0, T); H^{s-1}(\mathbb{S}))$ for the periodic or
non-periodic. Furthermore, the map $u_{0}\rightarrow u$ is continuous from a neighborhood of $u_{0}$ in $H^s
(\mathbb{S})$ into $C([0, T); H^s(\mathbb{S}))\cap C^1([0, T); H^{s-1}(\mathbb{S}))$.
\end{lemma}
In the periodic case and for $s>5/2$, the above lemma was proved by Ti$\check{g}$lay using Arnold's geometric 
framework \cite{FT}.  Under the assumptions of $s\geq3$, Ti$\check{g}$lay also proved the following global existence 
and properties of the strong solution.
\begin{lemma}
Let $u_{0}\in H^s(\mathbb{S})$ with $s\geq3$. Assume that $y_{0}=(1-\partial^2_{x})u_{0}\geq0$, then (\ref{2.1}) 
has a unique global strong solution
$$
u\in C([0, \infty); H^s(\mathbb{S}))\cap C^1([0, \infty); H^{s-1}(\mathbb{S}))
$$
with the initial data $u_{0}$. 
\end{lemma}
\begin{rem}(\cite{WY})
Let $u_{0}\in H^s(\mathbb{S})$ with $s\geq3$ and $y(t, \cdot)=u(t, \cdot)-\partial^2_{x}u(t, \cdot)$. Assume that 
$y_{0}=(1-\partial^2_{x})u_{0}\geq0$,  then for all $t>0$, $y(t, \cdot)\geq0$, $u(t, \cdot)>0$ and $|u_{x}(t, \cdot)|
\leq u(t, \cdot)$ on the unit circle.
\end{rem}
\begin{lemma}(\cite{FT})
The functionals $E_{2}$ and $E_{3}$ are conserved for the global strong solution $u$, that is for all $t\in [0, \infty)$
$$
\frac{d}{dt}\int_{\mathbb{S}}(u^2+u^2_{x})dx=0 \  and \  \frac{d}{dt}\int_{\mathbb{S}}\left(u^4+2u^2u^2_{x}-\frac{1}{3}
u^4_{x}\right)dx=0.
$$
Substituting the formula for $y$ in terms of $u$ into equation (\ref{CNV}) leads to the following equation
$$
u_{t}+u^2u_{x}+(1-\partial^2_{x})^{-1}\partial_{x}\left(u^3+\frac{3}{2}uu^2_{x}\right)+(1-\partial^2_{x})^{-1}
\left(\frac{u^3_{x}}{2}\right)=0.
$$
Recall that 
$$
u=(1-\partial^2_{x})^{-1}y=G*y,
$$
where $G(x)=\frac{\operatorname{ch}(\frac{1}{2}-x)}{2\operatorname{sh}(1/2)}$ is the Green function 
of the operator $(1-\partial^2_
{x})^{-1}$ for the periodic case, and $*$ denotes the convolution product on $\mathbb{S}$, defined by
$$
(f*g)(x)=\displaystyle\int_{\mathbb{S}}f(y)g(x-y)dy.
$$
\end{lemma}
In terms of this formulation, one can define the following periodic weak solutions.

\begin{de}\label{2.2}
Given initial data $u_{0}\in W^{1,3}(\mathbb{S})$, the function $u\in L^{\infty}_{loc}([0, T), W^{1,3}_{loc}
(\mathbb{S}))$ is called a periodic weak solution to the initial value problem (\ref{2.1}) if it satisfies the 
following identity:
\begin{equation}
\begin{array}{ll}
\displaystyle\int^T_{0}\int_{\mathbb{S}}\left[u\partial_{t}\phi+\frac{1}{3}u^3\partial_{x}\phi+G(x)*\left(u^3
+\frac{3}{2}uu^2_{x}\right)\partial_{x}\phi-G(x)*\left(\frac{u^3_{x}}{2}\right)\phi\right]dxdt\\[3mm]
\quad\quad\quad\quad\quad \quad\quad\quad\quad\quad\quad+\displaystyle\int_{\mathbb{S}}
u_{0}(x)\phi(0,x)dx=0,
\end{array}
\end{equation}
for any smooth test function $\phi(t, x)\in C^{\infty}_{c}\left([0, T)\times \mathbb{S}\right)$. If $u$ is a weak 
solution on $[0, T)$ for every $T>0$, then it is called a global periodic weak solution.
\end{de}
\begin{rem}
It is inferred from the Sobolev embedding $W^{1,3}_{loc}(\mathbb{S}) \hookrightarrow C^\alpha(\mathbb{S})$ 
with $0\leq\alpha\leq\frac{2}{3}$ that Definition \ref{2.2} precludes the admissibility of discontinuous shock 
waves as weak solutions.
\end{rem}
In \cite{WY}, the authors proved that there exists a unique global weak solution to Eq.(\ref{NV}).
Note that the periodic peakons to the Novikov equation are not strong solutions but weak solutions, 
the proof is in next section. 

\section{Periodic Peaked solutions}
In this section, it is verified that the periodic peakons (\ref{1.9}) to the Novikov equation (\ref{2.1}) are the global
periodic weak solution. Recall that the existence of periodic peakons for both the CH  and $\mu$-CH equations 
and the amplitudes are proportional to speed $c\in R$.
\begin{theorem}\label{thp3.1}
The periodic peaked functions of the form
\begin{equation}\label{p3.1}
u(x, t)=a\operatorname{ch}(\zeta),\ \zeta=\frac{1}{2}-(x-ct)+[x-ct].
\end{equation}
with $a^2=\frac{c}{\operatorname{ch}^2(1/2)}$are global periodic weak solutions to equation (\ref{NV}) in the sense of 
Definition \ref{2.2}.
\end{theorem}
\begin{rem}
Note that all peakons in (\ref{p3.1}) move with positive wave speed $c>0$. Each positive wave speed has a peakon 
and anti-peakon of opposite amplitudes: $a=\pm\frac{\sqrt{c}}{\operatorname{ch}(\frac{1}{2})}$. 
\end{rem}

\begin{proof}
 We identify $\mathbb{S}=[0, 1)$ and regard $u$ as periodic functions 
on $\mathbb{S}$. Clearly, $u(\zeta)$ is continuous for $\zeta\in \mathbb{S}$ with peak at $\zeta=0$. Notice that $u(\zeta)$ is 
smooth on $(0, 1)$  and for all $t\in \mathbb{R}^+$
\begin{equation}\label{p3.2}
u_{x}=-a\operatorname{sh}(\zeta),
\end{equation}
in the sense of periodic distribution. Moreover, (\ref{p3.2}) belongs to $L^{\infty}(\mathbb{S})$.  Then
\begin{equation}
\lim_{t\rightarrow0^+}\|u(t, \cdot)-u_{0}(\cdot)\|_{W^{1, \infty}(\mathbb{S})}=0.
\end{equation}
In view of (\ref{p3.2}), we have
\begin{equation}\label{p3.4}
u_{t}=ac\operatorname{sh}(\zeta)\in L^\infty(\mathbb{S}),\ t\geq0.
\end{equation}
Hence,  we deduce from integration by parts that, for every test function $\phi(t, x)\in C^\infty_{c}\left([0, \infty)\times
\mathbb{S} \right)$,
\begin{equation}\label{p3.5}
\begin{array}{ll}
\displaystyle\int^\infty_{0}\int_{\mathbb{S}}\left(u\partial_{t}\phi+\frac{1}{3}u^3\partial_{x}\phi\right)dxdt+\int_{\mathbb{S}}
u_{0}\phi(x, 0)dx \\[4mm]
\quad\quad\quad\quad=-\displaystyle\int^\infty_{0}\int_{\mathbb{S}}\phi(\partial_{t}u+u^2\partial_{x}u)dxdt\\[4mm]
\quad\quad\quad\quad=-\displaystyle\int^\infty_{0}\int_{\mathbb{S}}\phi\left((ac
-a^3)\operatorname{sh}(\zeta)-a^3
\operatorname{sh}^3(\zeta)\right)dxdt,
\end{array}
\end{equation}
by using (\ref{p3.2})-(\ref{p3.4}) and the following identity 
\begin{equation}\label{p3.6}
u^2\partial_{x}u=-a^3\operatorname{ch}^2(\zeta)\operatorname{sh}(\zeta)=-a^3\left(\operatorname{sh}(\zeta)
-\operatorname{sh}^3(\zeta)\right).
\end{equation}
On the other hand, 
\begin{equation}\label{p3.7}
\begin{array}{ll}
\displaystyle\int^\infty_{0}\int_{\mathbb{S}}\left[G(x)*\left(u^3+\frac{3}{2}u(\partial_{x}u)^2\right)\partial_{x}\phi-\frac{1}{2}G(x)*
(\partial_{x}u)^3\phi\right]dxdt\\[4mm]
\qquad=-\displaystyle\int^\infty_{0}\int_{\mathbb{S}}\phi G(x)*\left(3u^2\partial_{x}u+\frac{1}{2}(\partial_{x}u)^3\right)
dxdt\\[4mm]
\qquad\qquad\qquad-\displaystyle\int^\infty_{0}\int_{\mathbb{S}}\phi G(x)*\left(\frac{3}{2}u(\partial_{x}u)^2\right)dxdt.
\end{array}
\end{equation}
Combining (\ref{p3.2}) with (\ref{p3.6}), we obtain
$$
3u^2\partial_{x}u+\frac{1}{2}(\partial_{x}u)^3=-3a^3\operatorname{ch}^2(\zeta)\operatorname{sh}(\zeta)-\frac{1}{2}a^3
\operatorname{sh}^3(\zeta)=-3a^3\operatorname{sh}(\zeta)-\frac{7}{2}a^3\operatorname{sh}^3(\zeta)
$$
and 
$$
\frac{3}{2}u(\partial_{x}u)^2=\frac{3a^3}{2}\operatorname{ch}(\zeta)\operatorname{sh}^2(\zeta),
$$
which together with (\ref{p3.7}) lead to
\begin{equation}\label{p3.8}
\begin{array}{ll}
\displaystyle\int^\infty_{0}\int_{\mathbb{S}}\left[G(x)*\left(u^3+\frac{3}{2}u(\partial_{x}u)^2\right)\partial_{x}\phi-
\frac{1}{2}G(x)*(\partial_{x}u)^3\phi\right]dxdt\\[4mm]
\qquad\qquad=a^3\displaystyle\int^\infty_{0}\int_{\mathbb{S}}\phi G(x)*\left(3\operatorname{sh}(\zeta)+\frac{7}{2}
\operatorname{sh}^3(\zeta)\right)dxdt\\[4mm]
\quad\qquad\qquad\qquad\displaystyle-\frac{3a^3}{2}\int^\infty_{0}\int_{\mathbb{S}}\phi G_{x}(x)*\left(\operatorname
{ch}(\zeta)\operatorname{sh}^2(\zeta)\right)dxdt.
\end{array}
\end{equation}
Noticing from the definition of $G(x)$ for the periodic case that
$$
\partial_{x}G(x)=-\frac{\operatorname{sh}(1/2-x+[x])}{2\operatorname{sh}(1/2)}, \  x\in\mathbb{R},
$$
we obtain
\begin{equation}\label{p3.9}
\begin{array}{ll}
\qquad G(x)*\left(3\operatorname{sh}(\zeta)+\displaystyle\frac{7}{2}\operatorname{sh}^3(\zeta)\right)(t, x)\\[4mm]
=\displaystyle\frac{1}{2\operatorname{sh}(1/2)}\int_{\mathbb{S}}\operatorname{ch}(1/2-(x-y)+[x-y])\cdot
(3\operatorname{sh}(1/2-(y-ct)+[y-ct])\\[4mm]
\qquad\qquad\qquad\qquad \qquad+\displaystyle\frac{7}{2}\operatorname{sh}^3(1/2-(y-ct)+[y-ct]))dy
\end{array}
\end{equation}
and 
\begin{equation}\label{p3.10}
\begin{array}{ll}
\qquad G_{x}(x)*\left(\operatorname{ch}(\zeta)\operatorname{sh}^2(\zeta)\right)(t, x)\\[4mm]
=-\displaystyle\frac{1}{2\operatorname{sh}(1/2)}\int_{\mathbb{S}}\operatorname{sh}(1/2-(x-y)+[x-y])\cdot
\operatorname{ch}(1/2-(y-ct)+[y-ct])\\[4mm]
\qquad\qquad\qquad\qquad \qquad\cdot\displaystyle\operatorname{sh}^2(1/2-(y-ct)+[y-ct])dy.
\end{array}
\end{equation}
When $x>ct$, we split the right-hand side of (\ref{p3.9}) into the following three parts:
\begin{equation}\label{p3.11}
\begin{array}{ll}
\qquad G(x)*\left(3\operatorname{sh}(\zeta)+\displaystyle\frac{7}{2}\operatorname{sh}^3(\zeta)\right)(t, x)\\[4mm]
=\displaystyle\frac{1}{2\operatorname{sh}(1/2)}\left(\int^{ct}_{0}+\int^x_{ct}+\int^1_{x}\right)\operatorname{ch}
(1/2-(x-y)+[x-y])\\[4mm]
\qquad\displaystyle\cdot\left(3\operatorname{sh}(1/2-(y-ct)+[y-ct])+\frac{7}{2}\operatorname{sh}^3
(1/2-(y-ct)+[y-ct]))dy\right)\\[4mm]
=I_{1}+I_{2}+I_{3}.
\end{array}
\end{equation}
Using the identity $\operatorname{sh}(3x)=4\operatorname{sh}^3(x)+3\operatorname{sh}(x)$, we directly compute 
$I_{1}$ as follows:
\begin{equation}\label{p3.12}
\begin{array}{ll}
I_{1}=\displaystyle\frac{1}{2\operatorname{sh}(1/2)}\int^{ct}_{0}\operatorname{ch}(1/2-x+y)
\cdot\left(3\displaystyle
\operatorname{sh}(-1/2+ct-y)+\frac{7}{2}\operatorname{sh}^3(-1/2+ct-y)\right)dy\\[4mm]
\quad=\displaystyle\frac{1}{2\operatorname{sh}(1/2)}\bigg(\int^{ct}_{0}\frac{3}{8}
\operatorname{ch}(1/2-x+y)\cdot\operatorname{sh}(-1/2+ct-y)dy\\[4mm]
\qquad\qquad\qquad \displaystyle+\int^{ct}_{0}\frac{7}{8}\operatorname{ch}(1/2-x+y)
\cdot\operatorname{sh}(-3/2+3ct-3y)dy\bigg)\\[4mm]
\quad=\displaystyle\frac{1}{32\operatorname{sh}(1/2)}(-3ct\operatorname{sh}(x-ct)-
\frac{3}{2}\operatorname{ch}(1-x+ct)+
\frac{3}{2}\operatorname{ch}(1-x-ct))\\[4mm]
\quad+\displaystyle\frac{7}{96\operatorname{sh}(1/2)}\left(-\frac{3}{2}\operatorname{ch}
(1+x-ct)+\frac{3}{2}
\operatorname{ch}(1+x-3ct)-\frac{3}{4}\operatorname{ch}(2-x+ct)+\frac{3}{4}
\operatorname{ch}(2-x-3ct)\right)\\[4mm]
\quad=\displaystyle\frac{3}{64\operatorname{sh}(1/2)}\bigg(-2ct\operatorname{sh}(x-ct)
-\operatorname{ch}(1-x+ct)+
\operatorname{ch}(1-x-ct) \\[4mm]
\qquad\displaystyle-\frac{7}{3}\operatorname{ch}(1+x-ct)+\frac{7}{3}
\operatorname{ch}(1+x-3ct)-\frac{7}{6}\operatorname{ch}
(2-x+ct)+\frac{7}{6}\operatorname{ch}(2-x-3ct) \bigg),
\end{array}
\end{equation}
Similarly,
\begin{equation}\label{p3.13}
\begin{array}{ll}
I_{2}=\displaystyle\frac{1}{2\operatorname{sh}(1/2)}\int^{x}_{ct}\operatorname{ch}
(1/2-x+y)\cdot\left(3\displaystyle
\operatorname{sh}(1/2+ct-y)+\frac{7}{2}\operatorname{sh}^3(1/2+ct-y)\right)dy\\[4mm]
\quad=\displaystyle\frac{1}{2\operatorname{sh}(1/2)}\int^{x}_{ct}\operatorname{ch}
(1/2-x+y)\cdot\bigg(\frac{3}{8}
\operatorname{sh}(1/2+ct-y)+\frac{7}{8}\operatorname{sh}(3/2+3ct-3y)\bigg)dy\\[4mm]
\quad=\displaystyle\frac{3}{64\operatorname{sh}(1/2)}\bigg(2(x-ct)\operatorname{sh}
(1-x+ct)-\frac{7}{3}\operatorname{ch}
(2-3x+3ct)+\frac{7}{3}\operatorname{ch}(2-x+ct)\\[4mm]
\qquad\qquad\qquad\qquad\displaystyle-\frac{7}{6}\operatorname{ch}(1-3x+3ct)
+\frac{7}{6}\operatorname{ch}(1+x-ct) \bigg),
\end{array}
\end{equation}
and
\begin{equation}\label{p3.14}
\begin{array}{ll}
I_{3}=\displaystyle\frac{1}{2\operatorname{sh}(1/2)}\int^{1}_{x}\operatorname{ch}(1/2-x+y)
\cdot\left(3\displaystyle
\operatorname{sh}(1/2+ct-y)+\frac{7}{2}\operatorname{sh}^3(1/2+ct-y)\right)dy\\[4mm]
\quad=\displaystyle\frac{1}{2\operatorname{sh}(1/2)}\int^{1}_{x}\operatorname{ch}(1/2-x+y)
\cdot\bigg(\frac{3}{8}
\operatorname{sh}(1/2+ct-y)+\frac{7}{8}\operatorname{sh}(3/2+3ct-3y)\bigg)dy\\[4mm]
\quad=\displaystyle\frac{3}{64\operatorname{sh}(1/2)}\bigg(-2(1-x)\operatorname{sh}(x-ct
)-\operatorname{ch}(1-x-ct)
+\operatorname{ch}(1-x+ct)\\[4mm]
\qquad\displaystyle-\frac{7}{3}\operatorname{ch}(1+x-3ct)+\frac{7}{3}\operatorname{ch}(1-3x+3ct)
-\frac{7}{6}\operatorname{ch}(2-x-3ct)+\frac{7}{6}\operatorname{ch}(2-3x+3ct) \bigg).
\end{array}
\end{equation}
Plugging (\ref{p3.12}), (\ref{p3.13}) and (\ref{p3.14}) into (\ref{p3.11}), we deduce that for $x>ct$, 
\begin{equation}\label{p3.15}
\begin{array}{ll}
\qquad G(x)*\left(3\operatorname{sh}(\zeta)+\displaystyle\frac{7}{2}\operatorname{sh}^3
(\zeta)\right)(t, x)\\[4mm]
=\displaystyle\frac{3}{64\operatorname{sh}(1/2)}\bigg(2(x-ct)\operatorname{sh}(1-x+ct)
-2(1-x+ct)\operatorname{sh}(x-ct)\\[4mm]
\quad\displaystyle-\frac{7}{6}\operatorname{ch}(1+x-ct)+\frac{7}{6}\operatorname{ch}(2-x+ct)
-\frac{7}{6}\operatorname{ch}(2-3x+3ct)
+\frac{7}{6}\operatorname{ch}(1-3x+3ct)\bigg).
\end{array}
\end{equation}
On the other hand, when $x>ct$,  the right-hand side of (\ref{p3.10}) can be split into the 
following three parts:
\begin{equation}\label{p3.16}
\begin{array}{ll}
\qquad G_{x}(x)*\left(\operatorname{ch}(\zeta)\operatorname{sh}^2(\zeta)\right)(t, x)\\[4mm]
=-\displaystyle\frac{1}{2\operatorname{sh}(1/2)}\left(\int^{ct}_{0}+\int^x_{ct}+\int^1_{x}
\right)\operatorname{sh}(1/2-(x-y)+[x-y])\\[4mm]
 \qquad\cdot \operatorname{ch}(1/2-(y-ct)+[y-ct])\cdot\displaystyle\operatorname{sh}^2
 (1/2-(y-ct)+[y-ct])dy\\[4mm]
=J_{1}+J_{2}+J_{3}.
\end{array}
\end{equation}
For $J_{1}$, using the identity $2\operatorname{sh}^2(x)=\operatorname{ch}(2x)-1$, a 
direct calculation leads to
\begin{equation}\label{p3.17}
\begin{array}{ll}
J_{1}=-\displaystyle\frac{1}{2\operatorname{sh}(1/2)}\int^{ct}_{0}\operatorname{sh}(1/2-x+y)\cdot
\operatorname{ch}(1/2-ct+y)
\cdot\operatorname{sh}^2(1/2-ct+y)dy\\[4mm]
\quad=-\displaystyle\frac{1}{4\operatorname{sh}(1/2)}\int^{ct}_{0}\operatorname{sh}(1/2-x+y)
\cdot\operatorname{ch}(1/2-ct+y)
\cdot\left(\operatorname{ch}(1-2ct+2y)-1\right)dy\\[4mm]
\quad=-\displaystyle\frac{1}{32\operatorname{sh}(1/2)}\bigg(2ct\operatorname{sh}(x-ct)+\frac{1}{2}
\operatorname{ch}(2-x+ct)-
\frac{1}{2}\operatorname{ch}(2-x-3ct) \\[4mm]
\qquad\displaystyle -\operatorname{ch}(1-x+ct)+ \operatorname{ch}(1-x-ct)-\operatorname{ch}(1+x-ct)
+\operatorname{ch}
(1+x-3ct)\bigg).
\end{array}
\end{equation}
Similarly, we obtain
\begin{equation}\label{p3.18}
\begin{array}{ll}
J_{2}=-\displaystyle\frac{1}{2\operatorname{sh}(1/2)}\int^{x}_{ct}\operatorname{sh}(1/2-x+y)\cdot
\operatorname{ch}(1/2+ct-y)\cdot\operatorname{sh}^2(1/2+ct-y)dy\\[4mm]
\quad=-\displaystyle\frac{1}{4\operatorname{sh}(1/2)}\int^{x}_{ct}\operatorname{sh}(1/2-x+y)\cdot
\operatorname{ch}(1/2+ct-y)\cdot\left(\operatorname{ch}(1+2ct-2y)-1)\right)dy\\[4mm]
\quad=-\displaystyle\frac{1}{32\operatorname{sh}(1/2)}\bigg(-2(x-ct)\operatorname{sh}(1-x+ct)
-\frac{1}{2}\operatorname{ch}(1+x-ct)
+\frac{1}{2}\operatorname{ch}(1-3x+3ct) \\[4mm]
\qquad\qquad\qquad\qquad\qquad\displaystyle -\operatorname{ch}(2-3x+3ct)+
 \operatorname{ch}(2-x+ct)\bigg),
\end{array}
\end{equation}
and
\begin{equation}\label{p3.19}
\begin{array}{ll}
J_{3}=-\displaystyle\frac{1}{2\operatorname{sh}(1/2)}\int^{1}_{x}\operatorname{sh}(-1/2-x+y)
\cdot\operatorname{ch}(1/2+ct-y)
\cdot\operatorname{sh}^2(1/2+ct-y)dy\\[4mm]
\quad=-\displaystyle\frac{1}{4\operatorname{sh}(1/2)}\int^{1}_{x}\operatorname{sh}(-1/2-x+y)
\cdot\operatorname{ch}(1/2+ct-y)
\cdot\left(\operatorname{ch}(1+2ct-2y)-1\right)dy\\[4mm]
\quad=-\displaystyle\frac{1}{32\operatorname{sh}(1/2)}\bigg(2(1-x)\operatorname{sh}(x-ct)+\frac{1}{2}
\operatorname{ch}(2-x-3ct)-
\frac{1}{2}\operatorname{ch}(2-3x+3ct) \\[4mm]
\quad\qquad\displaystyle-\operatorname{ch}(1+x-3ct)+ \operatorname{ch}(1-3x+ct)-\operatorname{ch}
(1-x-ct)+\operatorname{ch}
(1-x+ct)\bigg).
\end{array}
\end{equation}
Plugging (\ref{p3.17}), (\ref{p3.18}) and (\ref{p3.19}) into (\ref{p3.16}), we deduce that for $x>ct$,
\begin{equation}\label{p3.20}
\begin{array}{ll}
\qquad G_{x}(x)*\left(\operatorname{ch}(\zeta)\operatorname{sh}^2(\zeta)\right)(t, x)\\[4mm]
=-\displaystyle\frac{1}{32\operatorname{sh}(1/2)}\bigg(2(1-x+ct)\operatorname{sh}(x-ct)-2(x-ct)
\operatorname{sh}(1-x+ct)\\[4mm]
\qquad\displaystyle-\frac{3}{2}\operatorname{ch}(1+x-ct)+\frac{3}{2}\operatorname{ch}(2-x+ct)-
\frac{3}{2}\operatorname{ch}
(2-3x+3ct)+\frac{3}{2}\operatorname{ch}(1-3x+3ct)\bigg).
\end{array}
\end{equation}
It follows from (\ref{p3.7}), (\ref{p3.8}), (\ref{p3.15}) and (\ref{p3.20}) that
\begin{equation}\label{p3.21}
\begin{array}{ll}
\quad\displaystyle\int^\infty_{0}\int^1_{ct}\left[G(x)*\left(u^3+\frac{3}{2}u(\partial_{x}u)^2
\right)\partial_{x}\phi-\frac{1}{2}G(x)*
(\partial_{x}u)^3\phi\right]dxdt\\[4mm]
=a^3\displaystyle\int^\infty_{0}\int^1_{ct}\phi G(x)*\left(3\operatorname{sh}(\zeta)+\frac{7}{2}
\operatorname{sh}^3(\zeta)\right)
-\frac{3}{2}\phi G_{x}(x)*\bigg(\operatorname{ch}(\zeta)\cdot\operatorname{sh}^2(\zeta)\bigg)dxdt\\[4mm]
=\displaystyle\frac{a^3}{8\operatorname{sh}(1/2)}\int^\infty_{0}\int^1_{ct}\phi\bigg(\operatorname{ch}
(2-x+ct)-\operatorname{ch}
(1+x-ct)\\[4mm]
\qquad\qquad\qquad\qquad\qquad+\operatorname{ch}(1-3x+3ct)-\operatorname{ch}(2-3x+3ct)\bigg)dxdt.
\end{array}
\end{equation}
Using the identity
$$
\operatorname{ch}(x+y)=\operatorname{ch}(x)\cdot\operatorname{ch}(y)+\operatorname{sh}(x)\cdot
\operatorname{sh}(y),
$$
we have
$$
\begin{array}{ll}
\operatorname{ch}(2-x+ct)=\operatorname{ch}(3/2)\cdot\operatorname{ch}(1/2-x+ct)+\operatorname{sh}
(3/2)\cdot\operatorname{sh}(1/2-x+ct)\\[4mm]
\operatorname{ch}(1+x-ct)=\operatorname{ch}(3/2)\cdot\operatorname{ch}(1/2-x+ct)-\operatorname{sh}
(3/2)\cdot\operatorname{sh}(1/2-x+ct)\\[4mm]
\operatorname{ch}(1-3x+3ct)=\operatorname{ch}(1/2)\operatorname{ch}(3/2-3x+3ct)-\operatorname{sh}
(1/2)\cdot\operatorname{sh}(1/2-x+ct)
\end{array}
$$
and
$$
\operatorname{ch}(2-3x+3ct)=\operatorname{ch}(1/2)\cdot\operatorname{ch}(3/2-3x+3ct)+\operatorname{sh}(1/2)
\cdot\operatorname{sh}(3/2-3x+3ct),
$$
which along with (\ref{p3.21}), we have 
\begin{equation}\label{p3.22}
\begin{array}{ll}
\quad\displaystyle\int^\infty_{0}\int^1_{ct}\left[G(x)*\left(u^3+\frac{3}{2}u(\partial_{x}u)^2
\right)\partial_{x}\phi-\frac{1}{2}G(x)*(\partial_{x}u)^3\phi\right]dxdt\\[4mm]
=\displaystyle\frac{a^3}{8\operatorname{sh}(1/2)}\int^\infty_{0}\int^1_{ct}\phi\bigg(2\operatorname{sh}(3/2)\cdot
\operatorname{sh}(1/2-x+ct)-2\operatorname{sh}(1/2)\cdot\operatorname{sh}(3/2-3x+3ct)\bigg)dxdt\\[4mm]
=\displaystyle\frac{a^3}{8\operatorname{sh}(1/2)}\int^\infty_{0}\int^1_{ct}\phi\bigg(8\operatorname{sh}^3(1/2)\cdot
\operatorname{sh}(1/2-x+ct)-8\operatorname{sh}(1/2)\cdot\operatorname{sh}^3(1/2-x+ct)\bigg)dxdt\\[4mm]
=a^3\displaystyle\int^\infty_{0}\int^1_{ct}\phi\left(\operatorname{sh}^2(1/2)\cdot\operatorname{sh}(1/2-x+ct)-
\operatorname{sh}^3(1/2-x+ct) \right)dxdt.
\end{array}
\end{equation}
In a similar manner, for the case of $x<ct$, we have
\begin{equation}\label{p3.23}
\begin{array}{ll}
\quad\displaystyle\int^\infty_{0}\int^{ct}_{0}\left[G(x)*\left(u^3+\frac{3}{2}u(\partial_{x}u)^2\right)
\partial_{x}\phi-\frac{1}{2}G(x)*(\partial_{x}u)^3\phi\right]dxdt\\[4mm]
\quad=a^3\displaystyle\int^\infty_{0}\int^{ct}_{0}\phi G(x)*\left(3\operatorname{sh}(\zeta)+\frac{7}{2}
\operatorname{sh}^3(\zeta)\right)-\frac{3}{2}\phi G_{x}(x)*\bigg(\operatorname{ch}(\zeta)
\cdot\operatorname{sh}^2(\zeta)\bigg)dxdt\\[4mm]
\quad=\displaystyle\frac{a^3}{8\operatorname{sh}(1/2)}\int^\infty_{0}\int^{ct}_{0}\phi
\bigg(-\operatorname{ch}(2+x-ct)
-\operatorname{ch}(1-x+ct)\\[4mm]
\qquad\qquad\qquad\qquad\qquad-\operatorname{ch}(1+3x-3ct)-\operatorname{ch}
(2+3x-3ct)\bigg)dxdt\\[4mm]
\quad=a^3\displaystyle\int^\infty_{0}\int^{ct}_{0}\phi\left(-\operatorname{sh}^2(1/2)\cdot
\operatorname{sh}(1/2+x-ct)+\operatorname{sh}^3(1/2+x-ct)\right)dxdt.
\end{array}
\end{equation}
Hence, associated with (\ref{p3.22}), we have
\begin{equation}\label{p3.23}
\begin{array}{ll}
\quad\displaystyle\int^\infty_{0}\int_{\mathbb{S}}\left[G(x)*\left(u^3+\frac{3}{2}u(\partial_{x}u)^2
\right)\partial_{x}\phi-\frac{1}{2}G(x)*(\partial_{x}u)^3\phi\right]dxdt\\[4mm]
\qquad=a^3\displaystyle\int^\infty_{0}\int_{\mathbb{S}}\phi\left(\operatorname{sh}^2(1/2)\cdot
\operatorname{sh}(\zeta)-\operatorname{sh}^3(\zeta)\right)dxdt.
\end{array}
\end{equation}
In vier of (\ref{p3.5}), (\ref{p3.7}) and (\ref{p3.23}), using the definition of $u$, we deduce that
\begin{equation}\label{p3.25}
\begin{array}{ll}
-(ac-a^3)\operatorname{sh}(\zeta)+a^3\operatorname{sh}^3(\zeta)-G(x)*\displaystyle\left(3u^2
\partial_{x}u+\frac{1}{2}(\partial_{x}u)^3\right)(t, x)\\[4mm]
\qquad\qquad\qquad\qquad\qquad\qquad-\displaystyle\frac{3}{2}G_{x}(x)*\left(u
(\partial_{x}u)^2\right)(t, x)\\[4mm]
\quad=-(ac-a^3)\operatorname{sh}(\zeta)+a^3\operatorname{sh}^2(1/2)\operatorname{sh}(\zeta)\\[4mm]
\quad=\left(-ac+a^3\operatorname{ch}^2(1/2)\right)\operatorname{sh}(\zeta)\\[4mm]
\quad=0
\end{array}
\end{equation}
for all $(x, t)\in\mathbb{R}\times\mathbb{R}^+$. Therefore, we conclude from (\ref{p3.5}), (\ref{p3.7}) 
and (\ref{p3.25}) that
$$
\begin{array}{ll}
\displaystyle\int^T_{0}\int_{\mathbb{S}}\left[u\partial_{t}\phi+\frac{1}{3}u^3\partial_{x}\phi+G(x)*
\left(u^3+\frac{3}{2}uu^2_{x}\right)\partial_{x}\phi-G(x)*\left(\frac{u^3_{x}}{2}\right)\phi\right]dxdt\\[3mm]
\quad\quad\quad\quad\quad \quad\quad\quad\quad\quad\quad+\displaystyle\int_{\mathbb{S}}u_{0}
(x)\phi(0,x)dx=0,
\end{array}
$$
for any smooth test function $\phi(t, x)\in C^{\infty}_{c}\left([0, \infty)\times \mathbb{S}\right)$, which 
completes the proof
of Theorem \ref{thp3.1}.
\end{proof}

\section{Stability of Periodic Peakons}

In this section, we main prove the stability of periodic peakons for Novikov equation.  
First of all, we give some properties of periodic peakons.
It is obvious that
\begin{equation}\label{PP}
    u(x,t)=\varphi_{c}(x,t)=\sqrt{c}\varphi(x-ct),
\end{equation}
where $\varphi(x)$ is given for $x\in [0,1]$ by
\begin{equation}\label{PP1}
    \varphi(x)=\frac{\operatorname{ch}(\frac{1}{2}-x)}{\operatorname{ch}(\frac{1}{2})},
\end{equation}
and extends periodically to the entire line. We still identify $\mathbb{S}$ with the interval
$[0,1)$ and view functions on $\mathbb{S}$ as periodic functions on the entire line of period one.

Notice that $\varphi(x)$ is continuous on the interval $(0,1)$ and has the peak at $x=0$.
In view of (\ref{PP1}), we obtain
\begin{equation}\label{Mm}
M_{\varphi}=\max_{x\in\mathbb{S}}{\varphi(x)}=\varphi(0)=1\quad and \quad m_{\varphi}=\min_
{x\in\mathbb{S}}{\varphi(x)}=\varphi(\frac{1}{2})=\frac{1}{\operatorname{ch}(\frac{1}{2})}
\end{equation}
Moreover, $\varphi(x)$ is smooth on $(0,1)$, and satisfies
\begin{equation}\label{D1}
\varphi_{x}(x)=-\frac{\operatorname{sh}(\frac{1}{2}-x)}{\operatorname{ch}(\frac{1}{2})},
\end{equation}
with
$$
\lim_{x\rightarrow 0}\varphi_{x}(x)= -\operatorname{th}(\frac{1}{2}), \quad and\quad \lim_{x\rightarrow 1}
\varphi_{x}(x)=\operatorname{th}(\frac{1}{2}).
$$
Using the formulation $\varphi_{xx}(x)=\varphi(x)-2\operatorname{th}(\frac{1}{2})\delta$ and integration
 by parts, we obtain
\begin{equation}\label{E1}
\begin{array}{ll}
E_{2}(\varphi_{c})&=\displaystyle c\int_0^1(\varphi^2+\varphi^{2}_{x})dx\\[3mm]
            &=\displaystyle c\int_0^1(\varphi^2-\varphi\varphi_{xx})dx\\[3mm]
            &=\displaystyle c\int_0^1\left(\varphi^2-(\varphi-2\operatorname{th}(\frac{1}{2})\delta)
            \varphi\right)dx\\[3mm]
            &=\displaystyle2c \operatorname{th}(\frac{1}{2})=2\sqrt{c}\operatorname{th}(\frac{1}{2})
            M_{\varphi_{c}},
\end{array}
\end{equation}
where $M_{\varphi_{c}}=\max\limits_{x\in\mathbb{S}}\{\varphi_{c}(x)\}$, $\delta(x)=
\left \{
   \begin{array}{l}
      \infty,\quad x=0,\\
      0,\quad x\neq0,
   \end{array}
   \right.$
with $\displaystyle\int_{\mathbb{R}}\delta dx=1$.

Next, we can compute $E_{3}(\varphi_{c})$ directly as follows:
\begin{equation}\label{E2}
\begin{array}{ll}
E_{3}(\varphi_{c})&=\displaystyle c^2\int_0^1(\varphi^4+2\varphi^2\varphi^{2}_{x}-\frac{1}{3}
            \varphi^{4}_{x})dx\\[3mm]
            &=\displaystyle\frac{c^2}{\operatorname{ch}^4(\frac{1}{2})}\int_{-\frac{1}{2}}^{\frac{1}{2}}
            \left(\operatorname{ch}^4(x)+2\operatorname{ch}^2(x)\operatorname{sh}^2(x)-\frac{1}{3}
            \operatorname{sh}^4(x)
            \right)dx\\[3mm]
            &=\displaystyle\frac{c^2}{\operatorname{ch}^4(\frac{1}{2})}\int_{-\frac{1}{2}}^{\frac{1}{2}}
            \left(\frac{2}{3}\operatorname{sh}(3x)\operatorname{sh}(x)+2\operatorname{sh}^2(x)
            +1\right)dx\\[3mm]
            &=\displaystyle\frac{c^2}{\operatorname{ch}^4(\frac{1}{2})}\int_{-\frac{1}{2}}^{\frac{1}{2}}
            (\frac{1}{3}\operatorname{ch}(4x)+\frac{2}{3}\operatorname{ch}(2x))dx\\[3mm]
            &=\displaystyle\frac{c^2}{6}\frac{\operatorname{sh}(2)+4\operatorname{sh}(1)}
            {\operatorname{ch}^4(\frac{1}{2})}.
\end{array}
\end{equation}
Now using the double angle formulas, more precisely, it is well-known that
$$
\operatorname{sh}(2x)=2\operatorname{sh}(x)\operatorname{ch}(x) \quad and \quad \operatorname{ch}(2x)=
2\operatorname{ch}^2(x)-1=1+2\operatorname{sh}^2(x).
$$
Replace these identities in the previous expression above, we can rewrite it as
\begin{equation}\label{3.7}
E_{3}[\varphi_{c}]=2c^2[\operatorname{th}(\frac{1}{2})-\frac{1}{3}\operatorname{th}^3(\frac{1}{2})].
\end{equation}

We state the following theorem which shows that the periodic peakons of Novikov equation are orbitally stable.
\begin{theorem}\label{thSt}
Let $\varphi_{c}$ be the periodic peakons defined in (\ref{1.9}) traveling with speed $c>0$. Assume that $u_{0}\in H^s(\mathbb{S})$
for $s\geq3$, with $y_{0}=u_{0}-\partial^2_{x}u_{0}\geq0$. For every $\epsilon>0$,
there exists a $\delta>0$ such that if
$$
\|u_{0}-\varphi_{c}\|_{H^1(\mathbb{S})}<\delta,
$$
then the correspongding solution $u(t)$ of (\ref{2.1}) with initial data $u(0)=u_{0}$ satisfies
$$
\sup_{t\geq0}\|u(t, \cdot)-\varphi_{c}(\cdot-\xi(t))\|_{H^1(\mathbb{S})}<\epsilon,
$$
where $\xi(t)\in\mathbb{R}$ is the maximum point of the function$u(t, \cdot)$.
\end{theorem}
\begin{rem}
Since the equation in (\ref{2.1}) is invariant under the inverse transform, namely, $-u$ is the solution of (\ref{2.1}) with initial 
data $-u_{0}$, if $u$ is the solution of (\ref{2.1}) with initial data $u_{0}$.  It implies that the periodic peakon $-\varphi_{c}$ 
is also orbitally stable in the sense of Theorem \ref{thSt} with the initial data $u_{0}(x)\in H^s(\mathbb{S}),\ s\geq3$ and $y_{0}
\leq0$ for any $x\in\mathbb{S}$.
\end{rem}

The proof of Theorem \ref{thSt} will proceed through a series of lemmas. First we consider the
expansion of the conservation
law $E_{2}$ around the peakon $\varphi_{c}$ in the $H^1(\mathbb{S})$-norm. 
The following lemma suggests that the error term in this expansion is given by $4\sqrt{c}\operatorname{th}
(\frac{1}{2})$ times the difference between $\varphi_{c}$ and the perturbed solution $u$ at
the point of their respective peaks.
\begin{lemma}\label{lem3.1}
For any $u\in H^1(\mathbb{S})$ and $\xi\in \mathbb{R}$, we have
\begin{equation}\label{3.8}
E_{2}(u)-E_{2}(\varphi_{c})=\|u-\varphi_{c}(\cdot-\xi)\|^2_{H^1(\mathbb{S})}+
4\sqrt{c}\operatorname{th}(\frac{1}{2})
\left(u(\xi)-M_{\varphi_{c}}\right).
\end{equation}
\end{lemma}
\begin{proof}
Using the formulation $\varphi_{xx}=\varphi-2\operatorname{th}(\frac{1}{2})\delta$, we compute
$$
\begin{array}{ll}
&\|u-\varphi_{c}(\cdot-\xi)\|^2_{H^1(\mathbb{S})}\\[3mm]
&=\displaystyle \int_{\mathbb{S}}(u(x)-\varphi_{c}(x-\xi))^2dx+\int_{\mathbb{S}}
(u_{x}(x)-\partial_{x}\varphi_{c}(x-\xi))^2dx\\[3mm]
&=E_{2}(u)+E_{2}(\varphi_{c}(\cdot-\xi))-2\displaystyle \int_{\mathbb{S}}u(x)\varphi_{c}
(x-\xi)dx-2\int_{\mathbb{S}}u_{x}(x)\partial_{x}
\varphi_{c}(x-\xi)dx\\[3mm]
&=E_{2}(u)+E_{2}(\varphi_{c})-2\sqrt{c}\displaystyle \int_{\mathbb{S}}u(x+\xi)
\varphi(x)dx+2\sqrt{c}\displaystyle \int_{\mathbb{S}}u(x+\xi)
\varphi_{xx}(x)dx\\[3mm]
&=\displaystyle E_{2}(u)+E_{2}(\varphi_{c})-4\sqrt{c}\operatorname{th}(\frac{1}{2})u(\xi).
\end{array}
$$
Since $E_{2}(\varphi_{c})=2\sqrt{c}\operatorname{th}(\frac{1}{2})M_{\varphi_{c}}$, we deduce that
$$
\|u-\varphi_{c}(\cdot-\xi)\|^2_{H^1(\mathbb{S})}=E_{2}(u)-E_{2}(\varphi_{c})+4\sqrt{c}\operatorname{th}
(\frac{1}{2})(M_{\varphi_{c}}-u(\xi)).
$$
\end{proof}

\begin{rem} \label{rem2.2}
It was discussed in \cite{AJ}  that for a wave profile $u\in H^1(\mathbb{S})$, the functional
$E_{2}(u)$ represents the kinetic energy. The formula (\ref{3.8}) indicates that if a wave has
energy and height close to the peakon's energy and height, then the whole shape of this wave
is close to that of the peakon. Also, among all waves of fixed energy, the peakon has the maximal height.
\end{rem}

The following lemma is crucial to establish the result of stability of periodic peakons.
\begin{lemma}\label{lem3.2}
For any positive $u\in H^s(\mathbb{S}),\ s\geq3$, define a function
$$
G_{u}:\{(M,m)\in \mathbb{R}^2|M\geq m>0\}\rightarrow\mathbb{R}
$$
by
\begin{equation}\label{3.9}
\begin{array}{ll}
G_{u}(M,m)&=\displaystyle(\frac{4}{3}M^2-\frac{1}{3}m^2)\\[3mm]
&\displaystyle\times\bigg(2m^2\ln(\frac{M+\sqrt{M^2-m^2}}{m})-2M\sqrt{M^2-m^2}-m^2\\[3mm]
&+E_{2}(u)\bigg)+\displaystyle\frac{4}{3}M(M^2-m^2)^{\frac{3}{2}}+m^2E_{2}(u)-E_{3}(u).
\end{array}
\end{equation}
Then
$$
G_{u}(M_{u},m_{u})\geq0,
$$
where $M_{u}=\max\limits_{x\in \mathbb{S}}\{u(x)\}$,\quad$m_{u}=\min\limits_{x\in \mathbb{S}}\{u(x)\}$.
\end{lemma}

\begin{proof}
Let $0<u(x)\in H^s(\mathbb{S})\subset C^2(\mathbb{S}),\ (s\geq3)$ and write $M=M_{u}=\max\limits_{x\in \mathbb{S}}\{u(x)\}$
and $m=m_{u}=\min\limits_{x\in \mathbb{S}}\{u(x)\}$. Let $\xi$ and $\eta$ be such that $u(\xi)=M$ and $u(\eta)=m$.
Following the similar computation as in \cite{LJ}, for the real function $g(x)$ which is defined by
\begin{equation}\label{3.10}
g(x)=
\left \{
   \begin{array}{l}
      u_{x}+\sqrt{u^2-m^2},\quad \xi<x\leq\eta,\\
      u_{x}-\sqrt{u^2-m^2},\quad \eta\leq x<\xi+1,
   \end{array}
   \right.
\end{equation}
and extended periodically to the entire line, we obtain
\begin{equation}\label{3.11}
\displaystyle \int_{\mathbb{S}}g^2(x)dx=2m^2\ln(\frac{M+\sqrt{M^2-m^2}}{m})-2M\sqrt{M^2-m^2}-m^2+E_{2}(u).
\end{equation}

On the other hand, inspired by the idea in \cite{XLIU}, we define the other real function $h(x)$ by
\begin{equation}\label{3.12}
h(x)=
\left \{
   \begin{array}{l}
      \displaystyle u^2+\frac{2}{3}u_{x}\sqrt{u^2-m^2}-\frac{1}{3}u^2_{x}-m^2,\quad \xi<x\leq\eta,\\[3mm]
      \displaystyle u^2-\frac{2}{3}u_{x}\sqrt{u^2-m^2}-\frac{1}{3}u^2_{x}-m^2,\quad \eta\leq x<\xi+1,
   \end{array}
   \right.
\end{equation}
and also extend it periodically to the entire line. We compute
$$
\begin{array}{ll}
&\displaystyle\int_{\mathbb{S}}h(x)g^2(x)\\[3mm]
&=\displaystyle\int^\eta_{\xi}(u^2+\frac{2}{3}u_{x}\sqrt{u^2-m^2}-\frac{1}{3}u^2_{x}-m^2)
(u_{x}+\sqrt{u^2-m^2})^2dx\\[3mm]
&\quad+\displaystyle\int^{\xi+1}_{\eta}(u^2-\frac{2}{3}u_{x}\sqrt{u^2-m^2}-\frac{1}{3}u^2_{x}-m^2)
(u_{x}-\sqrt{u^2-m^2})^2dx\\[3mm]
&=I_{1}+I_{2}.
\end{array}
$$
A direct calculation reveals that
$$
\begin{array}{ll}
I_{1}&=\displaystyle\int^\eta_{\xi}(u^2+\frac{2}{3}u_{x}\sqrt{u^2-m^2}-\frac{1}{3}u^2_{x}-m^2)
(u^2_{x}+2u_{x}\sqrt{u^2-m^2}+u^2-m^2)dx\\[3mm]
&=\displaystyle\int^\eta_{\xi}(u^4+2u^2u^2_{x}-\frac{1}{3}u^4_{x})dx+\frac{8}{3}\int^\eta_{\xi}u^2
u_{x}\sqrt{u^2-m^2}dx\\[3mm]
&\quad\displaystyle-\frac{2}{3}m^2\int^\eta_{\xi}u_{x}\sqrt{u^2-m^2}dx-m^2\int^\eta_{\xi}(u^2+u^2_{x})dx
-m^2\int^\eta_{\xi}g^2(x)dx.
\end{array}
$$
Using the identity
$$
\frac{d}{dx}(u(u^2-m^2)^{\frac{3}{2}})=u_{x}(u^2-m^2)^{\frac{3}{2}}+3u^2u_{x}\sqrt{u^2-m^2},
$$
then it is deduced that
$$
\begin{array}{ll}
u^2u_{x}\sqrt{u^2-m^2}&=u_{x}(u^2-m^2)^{\frac{3}{2}}+m^2u_{x}\sqrt{u^2-m^2}\\[3mm]
&=\displaystyle\frac{d}{dx}(u(u^2-m^2)^{\frac{3}{2}})-3u^2u_{x}\sqrt{u^2-m^2}+m^2u_{x}\sqrt{u^2-m^2},
\end{array}
$$
and hence
$$
\begin{array}{ll}
\displaystyle\frac{8}{3}\int^\eta_{\xi}u^2u_{x}\sqrt{u^2-m^2}dx&=\displaystyle
\frac{2}{3}\int^\eta_{\xi}\frac{d}{dx}(u(u^2-m^2)^{\frac{3}{2}})dx+\frac{2}{3}m^2
\int^\eta_{\xi}u_{x}\sqrt{u^2-m^2}dx\\[3mm]
&=\displaystyle-\frac{2}{3}M(M^2-m^2)^{\frac{3}{2}}+\frac{2}{3}m^2\int^\eta_{\xi}
u_{x}\sqrt{u^2-m^2}dx.
\end{array}
$$
It follows that
\begin{equation}\label{3.13}
\begin{array}{ll}
\displaystyle I_{1}=&\displaystyle\int^\eta_{\xi}(u^4+2u^2u^2_{x}-\frac{1}{3}u^4_{x})dx
-\frac{2}{3}M(M^2-m^2)^{\frac{3}{2}}\\[3mm]
&-m^2\displaystyle\int^\eta_{\xi}(u^2+u^2_{x})dx-m^2\int^\eta_{\xi}g^2(x)dx.
\end{array}
\end{equation}
Similarly, we can also deduce
$$
\begin{array}{ll}
\displaystyle I_{2}&=\displaystyle\int^{\xi+1}_{\eta}(u^2-\frac{2}{3}u_{x}\sqrt{u^2-m^2}-
\frac{1}{3}u^2_{x}-m^2)(u^2_{x}-2u_{x}\sqrt{u^2-m^2}+u^2-m^2)dx\\[3mm]
&=\displaystyle\int^{\xi+1}_{\eta}(u^4+2u^2u^2_{x}-\frac{1}{3}u^4_{x})dx-\frac{8}{3}
\displaystyle\int^{\xi+1}_{\eta}u^2u_{x}\sqrt{u^2-m^2}dx\\[3mm]
&\displaystyle+\frac{2}{3}m^2\int^{\xi+1}_{\eta}u_{x}\sqrt{u^2-m^2}dx-m^2
\displaystyle\int^{\xi+1}_{\eta}(u^2+u^2_{x})dx-m^2\int^{\xi+1}_{\eta}g^2(x)dx\\[3mm]
&=\displaystyle\int^{\xi+1}_{\eta}(u^4+2u^2u^2_{x}-\frac{1}{3}u^4_{x})dx
-\frac{2}{3}M(M^2-m^2)^{\frac{3}{2}}\\[3mm]&-m^2
\displaystyle\int^{\xi+1}_{\eta}(u^2+u^2_{x})dx-m^2\int^{\xi+1}_{\eta}g^2(x)dx,
\end{array}
$$
which along with (\ref{3.13}), we have
\begin{equation}\label{3.14}
\displaystyle \int_{\mathbb{S}}h(x)g^2(x)dx=-\frac{4}{3}M(M^2-m^2)^{\frac{3}{2}}-m^2E_{2}(u)+E_{3}(u)
-m^2\int_{\mathbb{S}}g^2(x)dx.
\end{equation}
In view of the positivity of the function $u$, a direct use of Young's inequality for $x\in\mathbb{S}$
gives rise to
$$
\begin{array}{ll}
h(x)&=\displaystyle u^2(x)\pm\frac{2}{3}u_{x}(x)\sqrt{u^2(x)-m^2}-\frac{1}{3}u^2_{x}(x)-m^2\\[3mm]
&\displaystyle\leq u^2(x)+\frac{1}{3}(u^2(x)-m^2)-m^2\\[3mm]
&\displaystyle=\frac{4}{3}(u^2(x)-m^2)\\[3mm]
&\displaystyle\leq\frac{4}{3}(M^2-m^2),
\end{array}
$$
which together with (\ref{3.14}) leads to
\begin{equation}\label{3.15}
\begin{array}{ll}
&\displaystyle-\frac{4}{3}M(M^2-m^2)^{3/2}-m^2E_{2}(u)+E_{3}(u)-m^2\int_{\mathbb{S}}g^2(x)dx\\[3mm]
&\displaystyle\leq \frac{4}{3}(M^2-m^2)\int_{\mathbb{S}}g^2(x)dx.
\end{array}
\end{equation}
Therefore, combining (\ref{3.11}) and (\ref{3.15}), we conclude that
\begin{equation}\label{3.16}
\begin{array}{ll}
0\leq&\displaystyle(\frac{4}{3}M^2-\frac{1}{3}m^2)\bigg(2m^2\ln(\frac{M+\sqrt{M^2-m^2}}{m})\\[3mm]
&-2M\sqrt{M^2-m^2}-m^2+E_{2}(u)\bigg)\\[3mm]
&+\displaystyle\frac{4}{3}M(M^2-m^2)^{3/2}+m^2E_{2}(u)-E_{3}(u).
\end{array}
\end{equation}
\end{proof}
\begin{rem}\label{rem3.2}
It is noted that the functionals $g$ and $h$ are zero when $u$ is replaced by the periodic
peakons $\varphi_{c}$.  Moreover, the function $G_{u(t)}$ depends on $u$ only through the two
conservation laws $E_{2}(u)$  and $E_{3}(u)$.  Hence, $G_{u(t)}=G_{u}$ is independent of time.
\end{rem}

In the next lemma, we will derive some properties
of the function $G_{u}(M,m)$ associated to the peakon $\varphi_{c}$.

\begin{lemma}\label{lem3.3}
For the peakon $\varphi_{c}$, we have
$$
\begin{array}{ll}
&\displaystyle G_{\varphi_{c}}(M_{\varphi_{c}},m_{\varphi_{c}})=0, \quad \frac{\partial G_{\varphi_{c}}}{\partial M}
(M_{\varphi_{c}},m_{\varphi_{c}})=0,\quad \frac{\partial G_{\varphi_{c}}}{\partial m}
(M_{\varphi_{c}},m_{\varphi_{c}})=0\\[3mm]
&\displaystyle\frac{\partial^2 G_{\varphi_{c}}}{\partial M^2}(M_{\varphi_{c}},m_{\varphi_{c}})=-\frac{32}{3}
c\operatorname{th}(\frac{1}{2}),\quad \frac{\partial^2 G_{\varphi_{c}}}{\partial m^2}(M_{\varphi_{c}},m_{\varphi_{c}})
=-\frac{16}{3}c\operatorname{th}(\frac{1}{2}),\\[3mm] &\displaystyle\frac{\partial^2 G_{\varphi_{c}}}
{\partial M\partial m}(M_{\varphi_{c}},m_{\varphi_{c}})=0,
\end{array}
$$
where $M_{\varphi_{c}}=\max\limits_{x\in \mathbb{S}}\{\varphi_{c}(x)\}$ and $m_{\varphi_{c}}=
\min\limits_{x\in \mathbb{S}}\{\varphi_{c}(x)\}$.
\end{lemma}
\begin{proof}
It is inferred from (\ref{D1}) that the function $g(x)$ corresponding to the peakon
$\varphi_{c}$ is identically zero. Hence
the inequality (\ref{3.15}) is an equality in the case of the peakon, which implies that
$G_{\varphi_{c}}(M_{\varphi_{c}},m_{\varphi_{c}})=0$.

We denote
$$
\begin{array}{ll}
J_{u}(M,m)&=\displaystyle\int_{\mathbb{S}}g^2(x)dx\\[3mm]
&=\displaystyle2m^2\ln(\frac{M+\sqrt{M^2-m^2}}{m})-2M\sqrt{M^2-m^2}-m^2+E_{2}(u).
\end{array}
$$
Then $G_{u}$ can be written as
$$
G_{u}(M,m)=(\frac{4}{3}M^2-\frac{1}{3}m^2)J_{u}+\frac{4}{3}M(M^2-m^2)^{3/2}+m^2E_{2}(u)-E_{3}(u).
$$
A straightforward computation gives
\begin{equation}\label{3.17}
\begin{array}{ll}
&\displaystyle\frac{\partial J_{u}}{\partial M}=-4\sqrt{M^2-m^2},\\[3mm]
&\displaystyle\frac{\partial J_{u}}{\partial m}=4m \ln(\frac{M+\sqrt{M^2-m^2}}{m})-2m,\\[3mm]
&\displaystyle\frac{\partial^2 J_{u}}{\partial m^2}=4 \ln(\frac{M+\sqrt{M^2-m^2}}{m})
-\frac{4M}{\sqrt{M^2-m^2}}-2.
\end{array}
\end{equation}
A further differentiation gives
$$
\begin{array}{ll}
\displaystyle\frac{\partial G_{u}}{\partial M}&=\displaystyle\frac{8}{3}MG_{u}+(\frac{4}{3}
M^2-\frac{1}{3}m^2)\frac{\partial J_{u}}{\partial M}+\frac{4}{3}(M^2-m^2)^{3/2}+4M^2
\sqrt{M^2-m^2}\\[3mm]
&=\displaystyle\frac{8}{3}MG_{u}-4\sqrt{M^2-m^2}(\frac{4}{3}M^2-\frac{1}{3}m^2)+\sqrt{M^2-m^2}
(\frac{16}{3}M^2-\frac{4}{3}m^2)\\[3mm]
&=\displaystyle\frac{8}{3}MJ_{u},\\[3mm]
\displaystyle\frac{\partial G_{u}}{\partial m}
&=\displaystyle-\frac{2}{3}mJ_{u}+(\frac{4}{3}M^2-\frac{1}{3}m^2)\frac{\partial J_{u}}
{\partial m}-4Mm\sqrt{M^2-m^2}+2mE_{2}(u)
\end{array}
$$
and
$$
\begin{array}{ll}
&\displaystyle\frac{\partial^2 G_{u}}{\partial M^2}=\frac{8}{3}J_{u}+\frac{8}{3}M
\frac{\partial J_{u}}{\partial M},\\[3mm]
&\displaystyle\frac{\partial^2 G_{u}}{\partial m^2}=-\frac{2}{3}J_{u}-\frac{4}{3}m
\frac{\partial J_{u}}{\partial m}+(\frac{4}{3}M^2-\frac{1}{3}m^2)
\frac{\partial^2 J_{u}}{\partial m^2}+\frac{8Mm^2-4M^3}{\sqrt{M^2-m^2}}+2E_{2}(u),\\[3mm]
&\displaystyle\frac{\partial^2 G_{u}}{\partial M\partial m}=\frac{8}{3}M
\frac{\partial J_{u}}{\partial m}.
\end{array}
$$
Recall that for the peakon $\varphi_{c}$,
\begin{equation}\label{3.18}
M_{\varphi_{c}}=\sqrt{c},\quad m_{\varphi_{c}}=\frac{\sqrt{c}}{\operatorname{ch}(\frac{1}{2})},\quad
E_{2}(\varphi_{c})=2c\operatorname{th}(\frac{1}{2}).
\end{equation}
It follows that
\begin{equation}\label{3.19}
\sqrt{M^2_{\varphi_{c}}-m^2_{\varphi_{c}}}=\sqrt{c}\operatorname{th}(\frac{1}{2}),
\quad and \quad \ln(\frac{M_{\varphi_{c}}
+\sqrt{M^2_{\varphi_{c}}-m^2_{\varphi_{c}}}}{m_{\varphi_{c}}})=\frac{1}{2},
\end{equation}
where the identity $\operatorname{sh}(x)+\operatorname{ch}(x)=e^{x}$ is used.
Substituting (\ref{3.18}) and (\ref{3.19})
into (\ref{3.17}), we have
\begin{equation}\label{3.20}
\begin{array}{ll}
\displaystyle\frac{\partial J_{\varphi_{c}}}{\partial M}(M_{\varphi_{c}},m_{\varphi_{c}})
=-4\sqrt{c}\operatorname{th}(\frac{1}{2}), \quad \frac{\partial J_{\varphi_{c}}}
{\partial m}(M_{\varphi_{c}},m_{\varphi_{c}})=0,\\[3mm]
\displaystyle\frac{\partial^2 J_{\varphi_{c}}}{\partial m^2}(M_{\varphi_{c}},m_{\varphi_{c}})
=-\frac{4}{\operatorname{th}(\frac{1}{2})}.
\end{array}
\end{equation}

On the other hand, related to the peakon, $J_{\varphi_{c}}(M_{\varphi_{c}},m_{\varphi_{c}})=0$, and hence, to complete
the proof of the lemma, it suffices to take $G_{u}
=G_{\varphi_{c}}$, $M=M_{\varphi_{c}}$ and $m=m_{\varphi_{c}}$ in the expressions for the partial derivatives of $G$ and
use (\ref{3.18}), (\ref{3.19}) and (\ref{3.20}).
\end{proof}
\begin{lemma}\label{lem3.4}\cite{LJ}
It holds that
\begin{equation}\label{3.21}
\max\limits_{x\in \mathbb{S}}|v(x)|\leq\sqrt{\frac{\operatorname{ch}(\frac{1}{2})}{2\operatorname{sh}(\frac{1}{2})}}
\|v\|_{H^1(\mathbb{S})},\quad v\in H^1(\mathbb{S}).
\end{equation}
Moreover, $\sqrt{\frac{\operatorname{ch}(\frac{1}{2})}{2\operatorname{sh}(\frac{1}{2})}}$ is the
best constant, and equality holds if and only if $v=\varphi_{c}(\cdot-\xi)$ for some $c>0$ and $\xi\in \mathbb{R}$,
that is, if and only if $v$ has the shape of a peakon.
\end{lemma}

\begin{rem}\label{rem3.3}
The above Lemma again indicates that among all waves of fixed energy, the peakon has maximal height. Moreover,
the inequality (\ref{3.21}) implies that the map $v\in H^1(\mathbb{R})\mapsto\max\limits_{x\in\mathbb{S}}\{v(x)\}\in
\mathbb{R}$ is continuous.
\end{rem}
\begin{lemma}\label{lem3.5}\cite{LJ}
If $u\in C([0, T), H^1(\mathbb{S}))$, then
$$
M_{u(t)}=\max\limits_{x\in\mathbb{S}}\{u(x,t)\}\quad and \quad m_{u(t)}=\min\limits_{x\in\mathbb{S}}\{u(x,t)\}
$$
are continuous for $t\in [0, \infty)$.
\end{lemma}
In fact,  for $t, s\in[0, \infty)$,
$$
\begin{array}{ll}
|M_{u(t)}-M_{u(s)}|&=|\max\limits_{x\in\mathbb{S}}u(x,t)-\max\limits_{x\in\mathbb{S}}u(x,s) |\\[3mm]
&\leq \max\limits_{x\in\mathbb{S}}|u(x,t)-u(x,s)|\\[3mm]
&\leq\sqrt{\frac{\operatorname{ch}(\frac{1}{2})}{2\operatorname{sh}(\frac{1}{2})}}\|u(\cdot, t)-u(\cdot, s)\|_{H^1(\mathbb{S})}.
\end{array}
$$
The continuity of $m_{u(t)}$ is similar.

\begin{lemma}\label{lem3.7}
Let $u\in C([0, \infty), H^s(\mathbb{S})), \ s\geq3$,  be a global periodic weak solution of the Novikov equation 
with the initial data $u_{0}$. Given a small neighborhood $\mathcal{U}$ of $(M_{\varphi_{c}},m_{\varphi_{c}})$
in $\mathbb{R}^2$, there is a $\delta>0$ such that for $t\in [0, \infty)$
\begin{equation}\label{3.30}
(M_{u(t)},m_{u(t)})\in\mathcal{U}, \quad if \quad\|u(\cdot,0)-\varphi_{c}\|_{H^1(\mathbb{S})}<\delta.
\end{equation}
\end{lemma}
\begin{proof}
The function $G_{u(t)}$ in Lemma\ref{lem3.2} depends on $u$ only through the two
conservation laws $E_{2}(u)$  and $E_{3}(u)$.  Hence, $G_{u(t)}=G_{u}$ is independent of time. 
It follows from Lemma \ref{lem3.3} that the function $G_{\varphi_{c}}(M, m)$ has a critical point with negative definite Hessian at 
$(M_{\varphi}, m_{\varphi})$ and $G_{\varphi_{c}}(M_{\varphi_{c}}, m_{\varphi_{c}})=0$.  Since $G_{\varphi_{c}}(M, m)$ and 
its Hessian are continuous, there is a neighborhood $\Gamma\subset D\subset\mathbb{R}^2$ around $(M_{\varphi_{c}}, 
m_{\varphi_{c}})$, where $D$ is 
the domain of the function $G_{u}(M, m)$ defined in Lemma\ref{lem3.2}, such that $G_{\varphi_{c}}$ is concave downward with 
curvature bounded away from zero. This property along with the boundary values of $G_{\varphi_{c}}$  implies that
$$
G_{\varphi_{c}}(M, m)\ne l<0,\quad  on \quad D/ \Gamma,
$$
where the constant $l>0$ depending on $\Gamma$.

Suppose $\omega\in H^s(\mathbb{S}), \ s\geq3$ is a small perturbation of $\varphi_{c}$ in the $H^1(\mathbb{S})$-norm such 
that $E_{i}[w]=E_{i}[\varphi_{c}]+\epsilon_{i}, i=2,3$. $E_{i}(u)=E_{i}(\varphi_{c})+\epsilon_{i}, \ i=2,3$. Then
$$
G_{u}(M,m)=G_{\varphi_{c}}(M,m)+(\frac{4}{3}M^2+\frac{2}{3}m^2)\epsilon_{1}-\epsilon_{2}.
$$
So $G_{\omega}$ is a small perturbation of $G_{\varphi_{c}}$. Notice that the effect of the perturbation near the point
$(M_{\varphi_{c}},m_{\varphi_{c}})$ can be made arbitrarily small by choosing $\epsilon_{1}$ and $\epsilon_{2}$
small, that is,
$$
G_{\omega}(M, m)=G_{\varphi_{c}}(M,m)+\mathcal{O}(\epsilon_{i}).
$$
Therefore, the set where $G_{\omega}\geq0$ near $(M_{\varphi_{c}},m_{\varphi_{c}})$ will be contained in a neighborhood
of $(M_{\varphi_{c}},m_{\varphi_{c}})$ provided that $\epsilon_{i}$ is small enough.

Now, let $\mathcal{U}$ be given as in the statement of the lemma. Shrinking $\mathcal{U}$ if necessary,
we infer that there exist a $\delta'>0$ such that for $u\in C\left([0, \infty), H^s(\mathbb{S})\right),\ s\geq3$ with
\begin{equation}\label{3.31}
|E_{2}(u)-E_{2}(\varphi_{c})|\leq \delta' \quad and \quad |E_{3}(u)-E_{3}(\varphi_{c})|\leq \delta',
\end{equation}
it holds that the set where $G_{u(t)}\geq0$ near $(M_{\varphi_{c}},m_{\varphi_{c}})$ is contained in $\mathcal{U}$
for each $t\in[0,+\infty)$. Lemma\ref{lem3.2} and \ref{lem3.5} suggest that $M_{u(t)}$ and $m_{u(t)}$ are continuous
functions in $t\in[0, \infty)$ and $G_{u(t)}(M_{u(t)},m_{u(t)})\geq0$ for $t\in[0, \infty)$. We deduce that for
$u$ satisfying (\ref{3.31}), if $(M_{u(0)},m_{u(0)})\in \mathcal{U}$, then
$$
(M_{u(t)},m_{u(t)})\in \mathcal{U}, \quad for\ t\in[0, \infty).
$$
However, the continuity of the conserved functionals $E_{i}:H^s(\mathbb{S})\mapsto\mathbb{R}, i=2,3$, 
shows that there is a $\delta>0$ such that (\ref{3.31}) holds for all
$u$ with
$$
\|u(\cdot,0)-\varphi_{c}\|_{H^1(\mathbb{S})}<\delta.
$$
Moreover, using the inequality (\ref{3.21}), choosing a smaller $\delta$ if necessary, we may also assume that
$(M_{u(0)},m_{u(0)})\in \mathcal{U}$ if $\|u(\cdot,0)-\varphi_{c}\|_{H^1(\mathbb{S})}<\delta$. Therefore, we end the
proof of the lemma.
\end{proof}
\paragraph{Proof of Theorem \ref{thSt}}
Let  $u\in C([0, \infty), H^s(\mathbb{S})), \ s\geq3$ be a periodic solution of the Novikov equation with initial $u_{0}$
and $\varepsilon>0$ be arbitrary. Take a neighborhood $\mathcal{U}$ of $(M_{\varphi_{c}},m_{\varphi_{c}})$ small enough
such that
$$
|M-M_{\varphi_{c}}|<\frac{\varepsilon^2}{8\sqrt{c}\operatorname{th}(\frac{1}{2})}, \quad if \quad (M,m)\in\mathcal{U}.
$$
Choose a $\delta>0$ as in Lemma\ref{lem3.7} so that for $t\in[0, \infty)$ (\ref{3.30}) holds. Taking a smaller $\delta$
if necessary we may also assume that
$$
|E_{2}(u)-E_{2}(\varphi_{c})|<\frac{\varepsilon^2}{2}, \quad if \quad \|u(\cdot,0)-\varphi_{c}\|_{H^1(\mathbb{S})}<\delta.
$$
We conclude after using Lemma\ref{lem3.1} that for $t\in[0,+\infty)$,
$$
\begin{array}{ll}
\|u(\cdot,t)-\varphi_{c}(\cdot-\xi(t))\|^2_{H^1(\mathbb{S})}&=E_{2}[u]-E_{2}[\varphi_{c}]-4\sqrt{c}\operatorname{th}(\frac{1}{2})
(u(\xi(t),t)-M_{\varphi_{c}})\\[3mm]
&\leq|E_{2}[u]-E_{2}[\varphi_{c}]|+4\sqrt{c}\operatorname{th}(\frac{1}{2})|M_{u(t)}-M_{\varphi_{c}}|\\[3mm]
&<\varepsilon^2.
\end{array}
$$
where $\xi(t)\in \mathbb{R}$ is any point where $u(\xi(t),t)=M_{u(t)}$. This completes the proof of Theorem\ref{thSt}.

\end{document}